\documentclass[12pt,a4paper,psamsfonts]{amsart}
\usepackage{amssymb,amscd,amsxtra,calc,enumerate}
\usepackage{cmmib57}
\usepackage[matrix,arrow,curve]{xy}

\newtheorem{thm}{Theorem}[section]

\newtheorem{proposition}[thm]{Proposition}
\newtheorem{claim}[thm]{Claim}
\newtheorem{corollary}[thm]{Corollary}
\newtheorem{conjecture}[thm]{Conjecture}

\newtheorem{counterexample}[thm]{Counterexample}

\theoremstyle{definition}
\newtheorem{definition}[thm]{Definition}
\newtheorem{remark}[thm]{Remark}
\newtheorem{notation}[thm]{Notation}

\newcommand{\pr}{\mathbb{P}}
\newcommand{\N}{\mathbb{N}}
\newcommand{\Z}{\mathbb{Z}}
\newcommand{\Q}{\mathbb{Q}}
\newcommand{\R}{\mathbb{R}}
\newcommand{\C}{\mathbb{C}}
\newcommand{\F}{\mathbb{F}}

\newcommand{\ND}{\operatorname{N}^1}
\newcommand{\NE}{\operatorname{NE}}
\newcommand{\Eff}{\operatorname{Eff}}
\newcommand{\Nef}{\operatorname{Nef}}

\newcommand{\Bl}{\operatorname{Bl}}

\newcommand{\Spec}{\operatorname{Spec}}

\newcommand{\Pic}{\operatorname{Pic}}
\newcommand{\sI}{\mathcal{I}}

\newcommand{\sO}{\mathcal{O}}
\newcommand{\sN}{\mathcal{N}}
\newcommand{\sE}{\mathcal{E}}

\newcommand{\cont}{\operatorname{cont}}
\newcommand{\vol}{\operatorname{vol}}

\usepackage[dvips]{color}

\title[Slope stability of Fano manifolds]
{Towards a criterion for slope stability of Fano manifolds along divisors}
\author{Kento Fujita}

\begin{document}
\maketitle
\begin{abstract}{\noindent 
We give a simple criterion for slope stability of Fano manifolds $X$ 
along divisors or smooth subvarieties. 
As an application, we show that $X$ is slope stable along 
an ample effective divisor $D\subset X$ 
unless $X$ is isomorphic to a projective space and $D$ is a hyperplane section. 
We also 
give counterexamples to Aubin's conjecture 
on the relation between the anticanonical volume 
and the existence of a K\"ahler-Einstein metric. 
Finally, we consider the case that $\dim X=3$; we give a complete answer for slope (semi)stability 
along divisors of Fano threefolds.} 
\end{abstract}

\section{Introduction}

Let $X$ be a Fano manifold, that is, a smooth projective variety such that 
the anticanonical divisor $-K_X$ of $X$ is ample. 
It has been conjectured that the $K$-polystability of $(X,-K_X)$ is 
equivalent to the existence of K\"ahler-Einstein metrics. However it is difficult to judge 
the $K$-(poly, semi)stability in general. 
In this article, we consider slope stability, which was introduced by Ross and Thomas (see \cite{RT}), 
that is weaker than $K$-stability but is easy to describe. For example, 
the case a Fano manifold is not slope (semi)stable along a smooth curve 
has been completely classified, see \cite{fjt} and \cite{ylee}.  

First, we give a simple criterion for slope stability of Fano manifolds along 
divisors (or smooth subvarieties). 

\begin{proposition}[see Proposition \ref{divslope} for detail]\label{divslope_intro}
For a Fano manifold $X$ and a divisor $D\subset X$, $X$ is slope stable 
$($resp.\ slope semistable$)$ along $D$ 
if and only if 
$\xi(D)>0$ $($resp. $\geq 0)$, where 
\[
\xi(D)=\vol_X(-K_X)
+\bigl(\epsilon\left(D\right)-1\bigr)\vol_X\bigl(K_X-\epsilon\left(D\right)D\bigr)
-\int_0^{\epsilon(D)}\vol_X(-K_X-xD)dx
\]
and $\epsilon(D)$ is the Seshadri constant of $D$ with respect to $-K_X$. 
\end{proposition}

As an application, we can investigate the case where $D\subset X$ is an ample divisor. 

\begin{thm}[{=Theorem \ref{ample} \eqref{ample1}}]\label{ample_intro}
For a Fano manifold $X$ and an ample divisor $D\subset X$, $X$ is slope stable along $D$ 
unless $X$ is isomorphic to a projective space and $D$ is a hyperplane section.
\end{thm}

We also construct the Fano manifolds which are not slope semistable along 
some divisors but have ``small" anticanonical volumes, 
which are counterexamples to the following Conjecture \ref{Aubin} 
(cf.\ Remark \ref{Aubinrmk}) 
on the relation between the anticanonical volume of $X$ and 
the existence of a K\"ahler-Einstein metric on $X$. 

\begin{conjecture}[{see also Remark \ref{Aubinrmk}}]\label{Aubin} 
Let $X$ be a Fano $n$-fold. 
If the anticanonical volume $\vol_X(-K_X)$ is less than 
$\bigl({(n+1)}^2/{2n}\bigr)^n$, then $X$ admits K\"ahler-Einstein metrics.
\end{conjecture}

\begin{counterexample}[=Corollary \ref{Aubincounter}]\label{Aubincounter_intro} 
For any $n\geq 4$, there exists a Fano $n$-fold $X$ such that 
$\vol_X(-K_X)=2(3^n-1)$ $($hence $2(3^n-1)<\bigl({(n+1)}^2/{2n}\bigr)^n$ 
holds if $n\geq 5)$ 
but $X$ does not admit K\"ahler-Einstein metrics.
\end{counterexample}

Finally, using the classification result \cite{MoMu}, we give the complete answer 
for slope (semi)stability of Fano threefolds along divisors:

\begin{thm}[=Theorem \ref{three}]\label{three_intro}
Let $X$ be a smooth Fano threefold.
\begin{enumerate}
\renewcommand{\theenumi}{\arabic{enumi}}
\renewcommand{\labelenumi}{$(\theenumi)$}
\item
$X$ is slope semistable along any effective divisor but 
there exists a divisor $D\subset X$ such that $X$ is not slope stable along $D$ 
if and only if $X$ is isomorphic to one of:
\[
{\pr}^3, {\pr}^1\times{\pr}^2, {\pr}_{{\pr}^1\times{\pr}^1}(\sO (0,1)\oplus\sO (1,0)), 
{\pr}^1\times{\pr}^1\times{\pr}^1, {\pr}^1\times S_m\,\, (1\leq m\leq 7).
\]
\item
There exists a divisor $D\subset X$ such that $X$ is not slope semistable along $D$ 
if and only if $X$ is isomorphic to one of:
\begin{eqnarray*}
{\Bl}_{\rm{line}}{\Q}^3, {\Bl}_{\rm{line}}{\pr}^3, {\pr}_{{\pr}^2}(\sO\oplus\sO (1)), 
{\pr}_{{\pr}^2}(\sO\oplus\sO (2)), \\
{\pr}^1\times{\F}_1, {\pr}_{{\F}_1}(\sO\oplus\sO (e+f)), 
{\pr}_{{\pr}^1\times{\pr}^1}(\sO\oplus\sO (1,1)).
\end{eqnarray*}
\end{enumerate}
\end{thm}

\bigskip

\noindent\textbf{Acknowledgements.}
The author is grateful to Doctor Yuji Odaka who pointed out Corollary \ref{ODK} 
after the author has been calculated the slope of special Fano manifolds (Proposition \ref{nonsemiprop}), 
and helped him to improve Proposition \ref{divslope}. 

The author is partially supported by JSPS Fellowships for Young Scientists.

\bigskip

\noindent\textbf{Notation and terminology.}
We always consider over the complex number field $\C$. 
A \emph{variety} means an irreducible and reduced scheme of finite type over $\Spec\C$. 
The theory of extremal contraction, we refer the readers to \cite{KoMo}.
For a projective variety $X$, let $\Eff(X)$ (resp.\ $\Nef(X)$) be the effective 
(resp.\ nef) cone 
which is defined as the cone in $\ND(X)$ spanned by the classes 
of effective (resp.\ nef) divisors on $X$. 
For a complete variety $X$, the Picard number of $X$ is denoted by $\rho_X$. 
For a smooth projective variety $X$ and a $K_X$-negative extremal ray $R\subset\overline{\NE}(X)$,
we define the \emph{length} $l(R)$ of $R$ by
\[
l(R):=\min\{(-K_X\cdot C)\mid C\text{ is a rational curve with } [C]\in R\},
\]
and we define \emph{a minimal rational curve of $R$} such that a rational curve $C\subset X$ with 
$[C]\in R$ and $(-K_X\cdot C)=l(R)$.

For an algebraic variety $X$ and a closed subscheme $Y\subset X$, denotes corresponding ideal sheaf $\sI_Y\subset\sO_X$, and 
$\Bl_Y\colon{\Bl}_Y(X)\rightarrow X$ or $\Bl_{\sI_Y}\colon{\Bl}_{\sI_Y}(X)\rightarrow X$ denotes the blowing up of $X$ along $Y$.

For algebraic varieties $X_1,\cdots,X_k$, we write the projection 
$p_{1,\cdots,t}\colon X_1\times\cdots\times X_k\rightarrow X_1\times\cdots\times X_t$.

We say $X$ is a \emph{Fano manifold} if $X$ is a smooth projective variety 
whose anticanonical divisor $-K_X$ is ample.
We note that if $X$ is a Fano manifold, then there is the canonical embedding $\Pic(X)\hookrightarrow \ND(X)$. For a Fano manifold $X$, let the \emph{Fano index} 
of $X$ be 
\[
\max\{r\in\Z_{>0}|-K_X\sim rL\text{ for some Cartier divisor }L\}.
\]
For a complete $n$-dimensional variety $X$ and a nef Cartier divisor 
(or a nef invertible sheaf) $D$ on $X$, 
the \emph{volume} of $D$ (denotes $\vol_X(D)$) means the self intersection number $(D^n)$ of $D$.

The symbol ${\Q}^n$ denotes a smooth hyperquadric in ${\pr}^{n+1}$.
The symbol ${\F}_1$ denotes the Hirzebruch surface having the $(-1)$-curve $e\subset\F_1$, 
and let $f\subset\F_1$ be a fiber of ${\pr}^1$-bundle $\F_1\rightarrow\pr^1$.
The symbol $S_m\, (1\leq m\leq 7)$ denotes a (smooth) del Pezzo surface $S$ 
(Fano $2$-fold) such that the anticanonical volume $\vol_S(-K_S)$ is equal to $m$.

\section{Slope stabilities of polarized varieties}

We recall slope stability of polarized varieties, which has been introduced by Ross and Thomas.
See \cite{RT} in detail.

\begin{definition}\label{slope_review}
Let $(X,L)$ be a polarized variety of $\dim X=n$, let $Z\subset X$ be a 
closed subscheme, let $\sigma\colon\hat{X}\rightarrow X$ 
be the blowing up of $X$ along $Z$ and let $E\subset\hat{X}$ be the Cartier divisor 
defined by $\sO_{\hat{X}}(-E)=\sigma^{-1}\sI_Z\cdot\sO_{\hat{X}}$.
\begin{itemize}
\item
Let $\epsilon(\sI_Z;(X,L))$ be the \emph{Seshadri constant of $Z$ with respect to $L$} 
(we often write 
$\epsilon(Z,X)$ or $\epsilon(Z)$ instead of $\epsilon(\sI_Z;(X,L))$ for simplicity), 
which is defined as follows:
\[
\epsilon(\sI_Z;(X,L)):=\max\{c\in{\R}_{>0}|{\sigma}^*L-cE \text{ is nef on }\hat{X}\}.
\]
\item
For $k$, $xk\in\N$ with $k\gg0$, we can write
\[
\chi(\hat{X},{\sigma}^*(kL)-xkE)=a_0(x)k^n+a_1(x)k^{n-1}+\cdots+a_n(x),
\]
where $a_i(x)\in\Q[x]$. 
Let $\mu_c(\sI_Z,L)$ be the \emph{slope of $Z$ 
with respect to $L$ and $c\in(0,\epsilon(Z)]$} (we often write 
$\mu_c(Z)$ instead of $\mu_c(\sI_Z,L)$ for simplicity), which is defined as follows:
\[
\mu_c(\sI_Z,L):=\frac{\int_0^c(a_1(x)+\frac{a'_0(x)}{2})dx}{\int_0^ca_0(x)dx}.
\]
We also define \emph{the slope of $X$ with respect to $L$} as
\[
\mu(X)=\mu(X,L):=\frac{a_1}{a_0},
\]
where $a_i\in\Q$ are defined by $\chi(X,rL)=a_0r^n+a_1r^{n-1}+\cdots+a_n$.
\end{itemize}
\end{definition}

\begin{definition}[slope (semi)stability]\label{stable_semistable}
Let $(X,L)$ (we often omit the polarization $L$) and $Z\subset X$ be as above.

$(X,L)$ is \emph{slope stabe} (resp.\ \emph{slope semistable}) \emph{along} $Z$ if:
\begin{itemize}
\item
\emph{slope semistability:}\\
$\mu_c(Z)\leq\mu(X)$ for all $c\in(0,\epsilon(Z)]$,
\item
\emph{slope stability:}\\
$\mu_c(Z)<\mu(X)$ for all $c\in(0,\epsilon(Z))$, and also for $c=\epsilon(Z)$ 
if $\epsilon(Z)\in\Q$ and global sections of $L^k\otimes{\sI}_Z^{k\epsilon(Z)}$ saturates 
for $k\gg0$.
\end{itemize}

For a polarized variety $(X,L)$ and a coherent ideal sheaf $\sI\subset\sO_X$ with $\sO_{\Bl_{\sI}(X)}(-E):={\Bl_{\sI}}^{-1}\sI\cdot\sO_{\Bl_{\sI}(X)}$, 
\emph{global sections of $L\otimes\sI$ saturates} if ${\Bl_{\sI}}^*L(-E)$ is spanned by 
global sections of $L\otimes\sI$. This condition is weaker than the condition 
such that $L\otimes\sI$ is globally generated. 
\end{definition}

\begin{remark}\label{pol_ryaku}
If $X$ is a Fano manifold, then we omit the polarization $(X, -K_X)$. 
More precisely, 
\emph{slope stability of $X$ along a closed subscheme $Z\subset X$} 
is nothing but slope stability of $(X, -K_X)$ along a closed subscheme $Z\subset X$. 
\end{remark}

The following is a fundamental result. 

\begin{thm}[{\cite{don}, \cite{RT}}]\label{KE_stable}
Let $(X, L)$ be a polarized manifold. If $(X, L)$ admits a K\"ahler metric with 
constant scalar curvature, then $(X, L)$ is slope semistable along any closed subscheme 
$Z\subset X$. 

In particular, for a Fano manifold $X$, if $X$ admits a K\"ahler-Einstein metric 
then $X$ is slope semistable along any closed subscheme $Z\subset X$. 
\end{thm}

\section{Slope stability of Fano manifolds along smooth subvarieties or divisors}

In this section, we fix the notation. 

\begin{notation}\label{nttn}

We set that $X$ is a Fano $n$-fold and $Z\subset X$ is a smooth subvariety of codimension $r\geq 2$ 
or an effective divisor (not necessary smooth). 
If $Z$ is an effective divisor on $X$,  we set $r:=1$. 
We set $\sigma:=\Bl_Z\colon\hat{X}\rightarrow X$. 
Let $E\subset\hat{X}$ be the divisor 
which satisfies $\sO_{\hat{X}}(-E)\simeq\sigma^{-1}\sI_Z$ 
(i.e., if $r\geq 2$ then $E$ is the exceptional divisor of $\sigma$, and if $r=1$ then $\sigma$ is the identity morphism 
and $E=Z$). 
\end{notation}

Under the notation, we consider slope stability of $X$ along $Z\subset X$.
We can show that 
\begin{eqnarray*}
a_0(x) & = & \frac{1}{n!}\vol_{\hat{X}}\left(\sigma^*\left(-K_X\right)-xE\right),\\
a_1(x) & = & \frac{1}{2\cdot (n-1)!}\left(-K_{\hat{X}}\cdot\left(\sigma^*\left(-K_X\right)-xE\right)^{n-1}\right),\\
\mu(X) & = & \frac{n}{2}
\end{eqnarray*}
by the weak Riemann-Roch formula (cf.\ \cite{ylee}). 
Thus for $0<c\leq\epsilon(Z)$, we have 
\begin{eqnarray*}
 &  & \mu_c(Z)<\mu(X)\\
 & \Leftrightarrow & \int_0^c(r-x)\left(E\cdot\left(\sigma^*\left(-K_X\right)-xE\right)^{n-1}\right)dx>0\\
 & \Leftrightarrow  & r\vol_X(-K_X)+(c-r)\vol_{\hat{X}}\left(\sigma^*\left(-K_X\right)-cE\right)-\int_0^c \vol_{\hat{X}}\left(\sigma^*\left(-K_X\right)-xE\right)dx>0.
\end{eqnarray*}

We set
\begin{eqnarray*}
\xi_c(Z) & := & r\vol_X(-K_X)+(c-r)\vol_{\hat{X}}\left(\sigma^*\left(-K_X\right)-cE\right)
-\int_0^c\vol_{\hat{X}}\left(\sigma^*\left(-K_X\right)-xE\right)dx.
\end{eqnarray*}
Since $\sigma^*(-K_X)-cE$ is ample for any $0<c<\epsilon(Z)$, we have 
\[
\frac{d}{dc}\left(\xi_c\left(E\right)\right)
=n(r-c)\left(E\cdot\left(\sigma^*\left(-K_X\right)-cE\right)^{n-1}\right)
\begin{cases}
>0 & (\text{if }0<c<\min\{r,\epsilon(Z)\}),\\
<0 & \left(\text{if }r<c<\epsilon(Z)\,\, \left(\text{if }\epsilon(Z)>r\right)\right).\\
\end{cases}
\]
Assume that $X$ is not slope stable along $Z$. Then $\xi_c(Z)\leq 0$ 
for some $0<c\leq\epsilon(Z)$. 
Hence $\epsilon(Z)>r$ and $\xi_c(Z)\geq\xi_{\epsilon(Z)}(Z)$ 
by the above argument. 
In particular, $\hat{X}$ is a Fano manifold 
and hence $\Nef(\hat{X})$ is a rational polyhedral cone spanned by semiample divisors 
(in particular, $\epsilon(Z)\in\Q_{>0}$ holds). 
Therefore we have the following result. 

\begin{proposition}\label{divslope}
Let $X$ be a Fano $n$-fold and $Z\subset X$ be a divisor or a smooth subvariety 
of codimension $r\geq 1$ $($if $Z$ is a divisor, then we set $r=1$$)$. 
Then $X$ is slope stable $($resp.\ slope semistable$)$ along $Z$ if and only if 
$\xi(Z)>0$ $($resp.\ $\geq 0)$, where 
\begin{eqnarray*}
\xi(Z) &:=& \xi_{\epsilon(Z)}(Z)\\
& = & r\vol_X(-K_X)+\left(\epsilon\left(Z\right)-r\right)
\vol_{\hat{X}}\left(\sigma^*\left(-K_X\right)-\epsilon\left(Z\right)E\right)\\
& - & \int_0^{\epsilon(Z)}\vol_{\hat{X}}\left(\sigma^*\left(-K_X\right)-xE\right)dx\\
& = & n\int_0^{\epsilon(Z)}(r-x)\left(E\cdot\left(\sigma^*\left(-K_X\right)-xE\right)^{n-1}\right)dx.
\end{eqnarray*}
\end{proposition}

\begin{remark}\label{icchi}
For a Fano $n$-fold $X$ and $Z\subset X$ a divisor or a smooth subvariety, 
the definition of slope stability of $X$ along $Z$ is equivalent to the definition 
in \cite{ylee} and \cite{fjt} by the above argument. 
\end{remark}

\begin{remark}\label{sesh}
If $X$ is not slope stable along $Z$, then $\epsilon(Z)>r$ holds by the above argument. 
This result has been already known in \cite[Lemma 2.10]{ylee}. 
In fact, Yuji Odaka pointed out to the author that 
if $X$ is not slope stable (resp.\ not slope semistable) along $Z$ then $\epsilon(Z)\geq r(n+1)/n$ (resp.\ $\epsilon(Z)>r(n+1)/n$) holds. 
See \cite[Proposition 4.4]{OS} for detail. 
\end{remark}

Now, we show that slope stability of Fano manifolds along smooth subvarieties 
can reduce to slope stability of Fano manifolds along divisors. 

\begin{proposition}\label{reduction}
Let $X$ be a Fano $n$-fold and let 
$Z\subset X$ be a smooth subvariety of codimension $r\geq 2$. 
Let $\sigma:=\Bl_Z\colon\hat{X}\rightarrow X$ and let $E\subset\hat{X}$ be 
the exceptional divisor of $\sigma$. 
If $X$ is not slope stable along $Z$, then $\hat{X}$ itself is a Fano $n$-fold and $\hat{X}$ is not slope semistable along $E$. 
\end{proposition}

\begin{proof}
We have already seen that $\hat{X}$ is a Fano manifold. We note that 
\[
r<\epsilon(Z,X)=\epsilon(E,\hat{X})+r-1.
\]
Hence we have 
\begin{eqnarray*}
n\xi(E) & = & \int_0^{\epsilon(E,\hat{X})}(1-x)
\left(E\cdot \left(-K_{\hat{X}}-xE\right)^{n-1}\right)dx\\
& = & \int_0^{\epsilon(Z,X)-(r-1)}
\left(E\cdot\left(\sigma^*\left(-K_X\right)+\left(1-r-x\right)E\right)^{n-1}\right)dx\\
& = & \int_{r-1}^{\epsilon(Z,X)}(r-x)
\left(E\cdot\left(\sigma^*\left(-K_X\right)-xE\right)^{n-1}\right)dx\\
& = & n\xi(Z)-\int_0^{r-1}(r-x)
\left(E\cdot\left(\sigma^*\left(-K_X\right)-xE\right)^{n-1}\right)dx\\
& < & 0,
\end{eqnarray*}
since $\sigma^*(-K_X)-xE$ is ample for any $0<x<r-1\,\, (<\epsilon(Z,X))$. 
\end{proof}

\section{First properties}

\subsection{Convexity of the volume function}

In this section, we consider slope stability of Fano manifolds 
in terms of the convexity of the volume function $\vol_{\hat{X}}(\sigma^*(-K_X)-xE)$ (under Notation \ref{nttn}). 

\begin{proposition}\label{convex}
We fix Notation \ref{nttn}. 
\begin{enumerate}
\renewcommand{\theenumi}{\arabic{enumi}}
\renewcommand{\labelenumi}{$(\theenumi)$}
\item\label{convex1}
If $\sN_{Z/X}^\vee$ $($the dual of the normal bundle$)$ is nef 
and $\epsilon(Z)\geq 2r$ holds, 
then $X$ is not slope stable along $Z$. 
\item\label{convex2}
Furthermore, $X$ is not slope semistable along $Z$ if we assume the assumption 
in \eqref{convex1} and one of the following holds: 
$r\geq 2$, $\epsilon(Z)>2r$ or $\sN_{Z/X}^\vee$ is ample.  
\item\label{convex3}
If $r=1$, $Z^2$ is a nonzero effective cycle and $\epsilon(Z)\leq 2$ holds, then $X$ is slope stable along $Z$. 
\end{enumerate}
\end{proposition}

We note that a vector bundle $\sE$ on a projective variety $Y$ is \emph{nef} (resp. \emph{ample}) if the corresponding tautological line bundle 
$\sO_{\pr_{Y}(\sE)}(1)$ on $\pr_{Y}(\sE)$ is nef (resp. ample). 

\begin{proof}
We can assume $\epsilon(Z)>r$ by Remark \ref{sesh}. We write $\epsilon:=\epsilon(Z)$ for simplicity. 
We define $f(x):=\vol_{\hat{X}}\left(\sigma^*\left(-K_X\right)-xE\right)$. 
Then we can write 
\[
\xi(Z)=f(0)+(\epsilon-r)f(\epsilon)-\int_0^\epsilon f(x)dx.
\]
We note that 
\begin{eqnarray*}
\frac{df}{dx}(x) & = & -n\left(E\cdot\left(\sigma^*\left(-K_X\right)-xE\right)^{n-1}\right)<0\quad (\text{for any }0<x<\epsilon),\\
\frac{d^2f}{dx^2}(x) & = & n(n-1)\left(E^2\cdot\left(\sigma^*\left(-K_X\right)-xE\right)^{n-2}\right).
\end{eqnarray*}
We recall that $\sO_{\hat{X}}(-E)|_{E}\simeq\sO_{\pr_{Y}(\sE)}(1)$. 
Hence $f(x)$ is a convex upward (resp.\ strictly convex upward) and strictly monotone decreasing function over an interval $(0$, $\epsilon)$ 
if $\sN_{Z/X}^\vee$ is nef (resp.\ ample). 
Then \eqref{convex1} and \eqref{convex2} 
follows immediately. 
The proof of \eqref{convex3} is same as those of \eqref{convex1} and \eqref{convex2}. 
\end{proof}

\subsection{Product cases}\label{prod_section}

We consider the case that a Fano manifold $X$ can be decomposed into 
the product $X=X_1\times X_2$. It is easy to show that both $X_1$ and $X_2$ are 
Fano manifolds, the vector space $\ND(X)$ is naturally decomposed into 
$\ND(X)=\ND(X_1)\oplus\ND(X_2)$ and the cones can be written as 
$\Eff(X)=\Eff(X_1)+\Eff(X_2)$, $\Nef(X)=\Nef(X_1)+\Nef(X_2)$ under the decomposition, 
respectively. We set $n_i:=\dim X_i$ $(i=1$, $2)$. 

\begin{proposition}\label{prod1}
Let $D_1\subset X_1$ be a divisor on $X_1$. 
Then slope stability $($resp.\ slope semistability$)$ 
of $X_1$ along $D_1$ is equivalent to slope stability $($resp.\ slope semistability$)$ 
of $X$ along $p_1^*D_1$.
\end{proposition}

\begin{proof}
Let $\epsilon_1:=\epsilon(D_1, X_1)$. 
Then we know that $\epsilon_1=\epsilon(p_1^*D_1, X)$. 
By the definition of $\xi(p_1^*D_1)$, we have 
\begin{eqnarray*}
\xi(p_1^*D_1) & = & \vol_X(-K_X)+(\epsilon_1-1)\vol_X(-K_X-\epsilon_1p_1^*D_1)
-\int_0^{\epsilon_1}\vol_X(-K_X-xp_1^*D_1)dx\\
 & = & \binom{n_1+n_2}{n_1}\cdot\vol_{X_2}(-K_{X_2})\Bigl\{\vol_{X_1}(-K_{X_1})\\
 & + & (\epsilon_1-1)\vol_{X_1}(-K_{X_1}-\epsilon_1D_1)
-\int_0^{\epsilon_1}\vol_{X_1}(-K_{X_1}-xD_1)dx\Bigr\}\\
 & = & \binom{n_1+n_2}{n_1}\vol_{X_2}(-K_{X_2})\cdot\xi(D_1).
\end{eqnarray*}
Therefore the signs of $\xi(D_1)$ and $\xi(p_1^*D_1)$ are same. 
\end{proof}

\begin{proposition}\label{prod2}
Let $D_i\subset X_i$ be divisors on $X_i$ for $i=1$, $2$. 
If $X_i$ is slope semistable along $D_i$ 
for any $i=1$, $2$, then $X$ is slope stable along $D:=p_1^*D_1+p_2^*D_2$. 
\end{proposition}

\begin{proof}
Let $\epsilon_i:=\epsilon(D_i, X_i)$ for $i=1$, $2$ and let $\epsilon:=\epsilon(D, X)$. 
We can show that $\epsilon=\min\{\epsilon_1$, $\epsilon_2\}$. 
We can assume $\epsilon>1$ by Remark \ref{sesh}. 
We note that 
\[
\frac{d}{dx}\left(\vol_{X_i}\left(-K_{X_i}-xD_i\right)\right)
=-n_i\left(D_i\cdot\left(-K_{X_i}-xD_i\right)^{n_i-1}\right)<0\quad(0<x<\epsilon)
\]
for any $i=1$, $2$. By the definition of $\xi(D)$, we have 
\begin{eqnarray*}
\xi(D) & = & \int_0^1\left(\vol_X\left(-K_X\right)-\vol_X\left(-K_X-xD\right)\right)dx\\
 & - & \int_1^\epsilon\left(\vol_X\left(-K_X-xD\right)-\vol_X\left(-K_X-\epsilon D\right)\right)dx\\
& = & \binom{n_1+n_2}{n_1}\Bigl\{\int_0^1\bigl(\vol_{X_1}(-K_{X_1})\vol_{X_2}(-K_{X_2}\bigr)\\
 & - & \vol_{X_1}(-K_{X_1}-xD_1)\vol_{X_2}(-K_{X_2}-xD_2)\bigr)dx\\
 & - & \int_1^\epsilon\bigl(\vol_{X_1}(-K_{X_1}-xD_1)\vol_{X_2}(-K_{X_2}-xD_2)\\
 & - & \vol_{X_1}(-K_{X_1}-\epsilon D_1)\vol_{X_2}(-K_{X_2}-\epsilon D_2)\bigr)dx\Bigr\} \\
& = & \binom{n_1+n_2}{n_1}\Bigl\{\int_0^1\left(\vol_{X_1}\left(-K_{X_1}\right)
-\vol_{X_1}\left(-K_{X_1}-xD_1\right)\right)\vol_{X_2}\left(-K_{X_2}\right)dx\\
 & + & \int_0^1\vol_{X_1}\left(-K_{X_1}-xD_1\right)\left(\vol_{X_2}\left(-K_{X_2}\right)
-\vol_{X_2}\left(-K_{X_2}-xD_2\right)\right)dx\\
 & - & \int_1^\epsilon\left(\vol_{X_1}\left(-K_{X_1}-xD_1\right)-\vol_{X_1}\left(-K_{X_1}-\epsilon D_1\right)\right)\vol_{X_2}\left(-K_{X_2}-\epsilon D_2\right)dx\\
 & - & \int_1^\epsilon\vol_{X_1}\left(-K_{X_1}-xD_1\right)\left(\vol_{X_2}\left(-K_{X_2}-xD_2\right)-\vol_{X_2}\left(-K_{X_2}-\epsilon D_2\right)\right)dx\Bigr\}
\end{eqnarray*}
\begin{eqnarray*}
 & > & \binom{n_1+n_2}{n_1}\Bigl\{\int_0^1\left(\vol_{X_1}\left(-K_{X_1}\right)
-\vol_{X_1}\left(-K_{X_1}-xD_1\right)\right)\vol_{X_2}\left(-K_{X_2}-D_2\right)dx\\
 & + & \int_0^1\vol_{X_1}\left(-K_{X_1}-D_1\right)\left(\vol_{X_2}\left(-K_{X_2}\right)
-\vol_{X_2}\left(-K_{X_2}-xD_2\right)\right)dx\\
 & - & \int_1^\epsilon\left(\vol_{X_1}\left(-K_{X_1}-xD_1\right)
-\vol_{X_1}\left(-K_{X_1}-\epsilon D_1\right)\right)\vol_{X_2}\left(-K_{X_2}-D_2\right)dx\\
 & - & \int_1^\epsilon\vol_{X_1}\left(-K_{X_1}-D_1\right)\left(\vol_{X_2}\left(-K_{X_2}-xD_2\right)-\vol_{X_2}\left(-K_{X_2}-\epsilon D_2\right)\right)dx\Bigr\}    \\
 & = & \binom{n_1+n_2}{n_1}\Bigl\{\vol_{X_1}(-K_{X_1}-D_1)\cdot\xi_\epsilon(D_2)
+\vol_{X_2}(-K_{X_2}-D_2)\cdot\xi_\epsilon(D_1)\Bigr\}\\
 & \geq & \binom{n_1+n_2}{n_1}\Bigl\{\vol_{X_1}(-K_{X_1}-D_1)\cdot\xi(D_2)
+\vol_{X_2}(-K_{X_2}-D_2)\cdot\xi(D_1)\Bigr\}\geq 0.
\end{eqnarray*}
Therefore $X$ is slope stable along $D$. 
\end{proof}

As a consequence of Propositions \ref{prod1} and \ref{prod2}, 
we have the following result. 

\begin{corollary}\label{prod_cor}
Let $X$ be a Fano manifold which is the product of Fano manifolds $X=\prod_{i=1}^mX_i$. Then $X$ is slope stable $($resp.\ slope semistable$)$ 
along any divisor if and only if $X_i$ is slope stable $($resp.\ slope semistable$)$ 
along any divisor for any $1\leq i\leq m$. 
\end{corollary}

\subsection{Length of extremal rays}

We show that if a Fano manifold $X$ is not slope stable along a divisor, 
then there exists an extremal ray of the length $\geq 2$.

\begin{proposition}\label{ray}
Let $X$ be a Fano manifold and $D\subset X$ ba a divisor. 
Assume $X$ is not slope stable along $D$. Then for any irreducible curve $C\subset X$, 
we have $(-K_X\cdot C)>(D\cdot C)$. In particular, there exists an extremal ray $R\subset\NE(X)$ such that $l(R)\geq 2$.
\end{proposition}

\begin{proof}
We have $\epsilon(D)>1$ by Remark \ref{sesh}. Hence we have 
\[
(D\cdot C)\leq(-K_X/\epsilon(D)\cdot C)<(-K_X\cdot C).
\]
Since $D$ is an effective divisor, there exists an extremal ray $R\subset\NE(X)$ with 
a minimal rational curve $[C]\in R$ such that $(D\cdot C)>0$. 
Therefore we have $(-K_X\cdot C)\geq 2$ since $(-K_X\cdot C)>(D\cdot C)$ holds.
\end{proof}

\subsection{Slope stability of Fano manifolds along nef divisors}

\begin{thm}\label{ample}
Let $X$ be a Fano $n$-fold and $D\subset X$ be a divisor. 
\begin{enumerate}
\renewcommand{\theenumi}{\arabic{enumi}}
\renewcommand{\labelenumi}{$(\theenumi)$}
\item\label{ample1}
If $D$ is an ample divisor, then $X$ is slope stable along $D$ 
unless $X$ is isomorphic to a projective space and $D$ is a hyperplane section.
\item\label{ample2}
If $D$ is a nef divisor and $(D^i\cdot (-K_X-\epsilon(D)D)^{n-i})=0$ 
for any $1\leq i\leq \epsilon(D)-1$ $($resp.\ $1\leq i<\epsilon(D)-1$$)$, 
then $X$ is slope stable $($resp.\ slope semistable$)$ along $D$. 
\end{enumerate}
\end{thm}

\begin{proof}
First, we consider the case that $-K_X$ and $D$ are numerically proportional 
(i.e., there exists a positive rational number $t$ such that $-K_X\equiv tD$). 
We note that $t\leq n+1$ and the equality holds if and only if $X\simeq\pr^n$ and $D\in|\sO_{\pr^n}(1)|$ by \cite{KO}. 
In this case we have 
\begin{eqnarray*}
\xi(D)=\vol_X(D)\Bigl\{t^n-\int_0^t(t-x)^n dx\Bigr\}=\vol_X(D)t^n\Bigl(1-\frac{t}{n+1}\Bigr)\geq 0, 
\end{eqnarray*}
and equality holds if and only if $t=n+1$. Therefore we have proved the theorem for the case $-K_X$ and $D$ are numerically proportional. 

Now we consider the case that $-K_X$ and $D$ are not numerically proportional. We can assume $\epsilon(D)>1$ by Remark \ref{sesh}. 
Let $P\subset\ND(X)$ be the $2$-dimensional vector subspace 
spanned by $[-K_X]$ and $[D]$ (the classes of $-K_X$ and $D$ in $\ND(X)$). 
We take $[H_1],[H_2]\in P\cap\Nef(X)$ such that $P\cap\Nef(X)=\R_{\geq 0}[H_1]+\R_{\geq 0}[H_2]$. 
Then after interchanging $H_1$ and $H_2$, if necessary, we can write $-K_X\equiv p_1 H_1+p_2 H_2$ and $D\equiv q_1 H_1+q_2 H_2$, 
where $p_1$, $p_2$, $q_1 >0$, $q_2\geq 0$ 
and $1/{\epsilon(D)}=q_1/p_1>q_2/p_2$ holds. We note that $D$ is ample if and only if $q_2>0$. 
We also note that there exists extremal rays $R_1, R_2\subset\NE(X)$ such that $(H_i\cdot R_j)=0$ if and only if $i\neq j$ holds 
where $1\leq i$, $j\leq 2$, since the class of $-K_X$ lives in the interior of $\Nef(X)$. 
We choose minimal rational curves $C_1$ and $C_2$ of $R_1$ and $R_2$, respectively. 
We have $(-K_X\cdot C_i)\leq n$ for any $i=1$, $2$ by \cite{CMSB}. 
If $q_i>0$ then we have $p_i/q_i\leq n$ since $(-K_X\cdot C_i)\leq n$, 
$(D\cdot C_i)\geq 1$ and $p_i/q_i=(-K_X\cdot C_i)/(D\cdot C_i)$ holds. 
Then we can show that 

\begin{eqnarray*}
\xi(D) & = & \vol_X(p_1H_1+p_2H_2)+\Bigl(\frac{p_1}{q_1}-1\Bigr)
\vol_X\Bigl(\bigl(p_2-q_2\frac{p_1}{q_1}\bigr)H_2\Bigr)\\
& - & \int_0^{\frac{p_1}{q_1}}\vol_X\left(\left(p_1-q_1x\right)H_1+\left(p_2-q_2x\right)H_2\right)dx\\
& = & (H_2^n)p_2^n\Bigl\{1+\Bigl(\frac{p_1}{q_1}-1\Bigr)\Bigr(1-\frac{q_2}{p_2}\frac{p_1}{q_1}\Bigr)^n-\int_0^{\frac{p_1}{q_1}}\Bigl(1-\frac{q_2}{p_2}x\Bigr)^n dx\Bigr\}\\
& + & \sum_{i=1}^{n-1}(H_1^i\cdot H_2^{n-i})p_1^ip_2^{n-i}\binom{n}{i}\Bigl\{1-\int_0^{\frac{p_1}{q_1}}\Bigl(1-\frac{q_1}{p_1}x\Bigr)^i\Bigl(1-\frac{q_2}{p_2}x\Bigr)^{n-i}dx\Bigr\}\\
& + & (H_1^n)p_1^n\Bigl\{1-\int_0^{\frac{p_1}{q_1}}\Bigl(1-\frac{q_1}{p_1}x\Bigr)^n dx\Bigr\}. 
\end{eqnarray*}

We denote the coefficient of $(H_1^i\cdot H_2^{n-i})$ by $M_i$. 
We claim that $(H_1^i\cdot H_2^{n-i})\geq 0$ for any $0\leq i\leq n$ and $(H_1^i\cdot H_2^{n-i})>0$ for some $i$ since $H_1$ and $H_2$ are nef and $[-K_X]\in P$. 

First, we consider the case $D$ is ample. It is enough to show $M_i>0$ for any $i$ by the above claim. 
We have 
\begin{eqnarray*}
M_0 &=& p_2^n\Bigl\{1-\frac{1}{n+1}\frac{p_2}{q_2}+\Bigl(1-\frac{q_2}{p_2}\frac{p_1}{q_1}\Bigr)^n\Bigl(\frac{p_1}{q_1}-1+\frac{1}{n+1}\bigl(\frac{p_2}{q_2}-\frac{p_1}{q_1}\bigr)\Bigr)\Bigr\}>0,\\
M_n &=& p_1^n\Bigl(1-\frac{1}{n+1}\frac{p_1}{q_1}\Bigr)>0,\\
M_i &>& \binom{n}{i}p_1^ip_2^{n-i}\Bigl\{1-\int_0^{\frac{p_1}{q_1}}\bigl(1-\frac{p_2}{q_2}x\bigr)^n dx\Bigr\}\\
&=&\binom{n}{i}p_1^ip_2^{n-i}\frac{p_2}{q_2(n+1)}\Bigl\{(n+1)\frac{q_2}{p_2}-1+\Bigl(1-\frac{q_2}{p_2}\frac{p_1}{q_1}\Bigr)^{n+1}\Bigr\}>0\quad (0<i<n).
\end{eqnarray*}
Thus we have proved the theorem for the case $D$ is ample. 

Now, we consider the case that $D$ is not ample. Since $q_2=0$, we have 
\[
\xi(D)=\sum_{i=1}^n(H_1^i\cdot H_2^{n-i})\binom{n}{i}p_1^ip_2^{n-i}\Bigl(1-\frac{1}{i+1}\frac{p_1}{q_1}\Bigr).
\]
Therefore we have $\xi(D)>0$ (resp.\ $\geq 0$) if $(H_1^i\cdot H_2^{n-i})=0$ 
for any $1\leq i\leq p_1/q_1-1$ (resp.\ $1\leq i<p_1/q_1-1$) 
by the same argument of the case $D$ is ample. 
\end{proof}

As an immediate corollary of Theorem \ref{ample}, we get Odaka's result: 

\begin{corollary}[Odaka]\label{ODK}
Let $X$ be a Fano manifold with the Picard number $\rho_X=1$ and let 
$D\subset X$ be a divisor. 
Then $X$ is slope stable along $D$ unless $X$ is isomorphic to a 
projective space and $D$ is a hyperplane section. 
\end{corollary}

\begin{proof}
The divisor $D$ is ample since $\rho_X=1$. Hence the assertion is obvious from 
Theorem \ref{ample} \eqref{ample1}. 
\end{proof}

\begin{remark}\label{nef_destabilize}
There exists a Fano $n$-fold $X$ and a nef effective divisor $D\subset X$ such that $X$ is not slope semistable along $D$. 
For example, let $X$ be the Fano manifold obtained by the blowing up of 
the $n$-dimensional projective space along a (reduced) point 
and $D$ be the strict transform of a hyperplane passing through the center of the blowing up. 
Then $D$ is a nef divisor and 
\[
\vol_X(-K_X-xD)=(n+1-x)^n-(n-1-x)^n
\]
holds. 
Hence we have 
\[
\xi(D)=\frac{2(n-1)}{n+1}\{n\cdot2^{n-1}-(n-1)^{n-1}\},
\]
which takes a negative value if $n\geq 5$. 
\end{remark}

\section{Examples and applications}

\subsection{Projective spaces}\label{pn}

Let $Z\subset\pr^n$ be a linear subspace of codimension $r\geq 1$. 

If $r=1$, then 
\[
\xi(Z)=(n+1)^n-\int_0^{n+1}(n+1-x)^ndx=0
\]
(see also the proof of Theorem \ref{ample}).

We consider the case $r\geq 2$. Let the blowing up of $\pr^n$ along $Z$ be 
$\sigma\colon\hat{X}\rightarrow\pr^n$ and the exceptional divisor be $E$. 
We can show that $\epsilon(Z)=n+1$, 
$E\simeq\pr^{n-r}\times\pr^{r-1}$, 
$\sN_{E/\hat{X}}\simeq\sO_{\pr^{n-r}\times\pr^{r-1}}(1, -1)$ and 
$\sO_{\hat{X}}(-K_{\hat{X}})|_E\simeq\sO_{\pr^{n-r}\times\pr^{r-1}}(n-r+2, r-1)$. 
Hence 
\begin{eqnarray*}
\xi(Z) & = & n\int_0^{n+1}(r-x)\vol_{\pr^{n-r}\times\pr^{r-1}}(\sO_{\pr^{n-r}\times\pr^{r-1}}(n+1-x, x))dx\\
 & = & n\binom{n-1}{r-1}\int_0^{n+1}(r-x)(n+1-x)^{n-r}x^{r-1}dx=0
\end{eqnarray*}
by a simple calculation. Therefore, we have the following: 

\begin{proposition}\label{proj}
The projective space $\pr^n$ is not slope stable 
but slope semistable along any linear subspace. 
\end{proposition}

In fact, it is well known that 
the $n$-dimensional projective space admits a K\"ahler-Einstein metrics; 
the Fubini-Study metric.

\subsection{Surfaces}\label{surf}

In Section \ref{surf}, we consider the case such that the dimension is equal to two. 

\begin{proposition}\label{dp}
Let $S$ be a del Pezzo surface, that is, $S$ is a Fano manifold with $\dim S=2$. 
\begin{enumerate}
\renewcommand{\theenumi}{\arabic{enumi}}
\renewcommand{\labelenumi}{$(\theenumi)$}
\item\label{dp1}
$S$ is slope semistable along any curve but 
there exists a curve $C\subset S$ such that $S$ is not slope stable along $C$ 
if and only if $S$ is isomorphic to either $\pr^2$ or $\pr^1\times\pr^1$.
\item\label{dp2}
There exists a curve $C\subset S$ such that $S$ is not slope semistable along $C$ 
if and only if $S$ is isomorphic to $\F_1$. 
\end{enumerate}
\end{proposition}

\begin{proof}
If $\vol_S(-K_S)\leq 7$, then we know that any extremal ray $R\subset\NE(S)$ 
satisfies that $l(R)=1$. 
Hence $S$ is slope stable along any curve by Proposition \ref{ray}. 
If $S=\pr^2$ or $\pr^1\times\pr^1$, then the assertion 
\eqref{dp1} in Proposition \ref{dp} 
holds by Theorem \ref{ample} \eqref{ample1} and Corollary \ref{prod_cor}.
If $S=\F_1$, then $S$ is not slope semistable along $e\subset\F_1$ by 
Propositions \ref{proj} and \ref{reduction}. 
\end{proof}

\begin{remark}\label{Tian}
In fact, Tian \cite{tian} proved that $S$ does \emph{not} admit 
K\"ahler-Einstein metrics if and only if 
$S$ is isomorphic to $\F_1$ or $S_7$. 
\end{remark}

\subsection{Non-slope-semistable examples}\label{nonsemi}
Let $Z$ ba a Fano $(n-1)$-fold of $\rho_Z=1$ and the Fano index $t\geq 2$. 
Let $\sO_Z(1)$ be the ample generator of $\Pic(Z)$. 
We note that $t\leq n$, see \cite{KO}. 

We set $X:=\pr_Z(\sO_Z\oplus\sO_Z(s))\xrightarrow{\pi}Z$ with $t>s>0$. 
We denote the section of $\pi$ with $\sN_{E/X}\simeq\sO_Z(-s)$ by $E\subset X$. 
Then it is easy to show that $X$ is a Fano $n$-fold which satisfies that  
\[
\vol_X(-K_X)=\frac{(t+s)^n-(t-s)^n}{s}\vol_Z(\sO_Z(1))
\]
and
\[
\NE(X)=\R_{\geq 0}[f]+\R_{\geq 0}[e],
\]
where $f$ is a fiber of $\pi$ and $e\subset E$ is an arbitrary 
irreducible curve in $E$. 
Then we can show that $\epsilon(E)=2$. 
Hence we have the following result by Proposition \ref{convex} \eqref{convex2}. 

\begin{proposition}\label{nonsemiprop}
$X$ is not slope semistable along $E$. 
\end{proposition}

As a corollary, we give the following counterexample. 

\begin{corollary}[Counterexamples to Conjecture \ref{Aubin}]\label{Aubincounter}
For any $n\geq 4$, there exists a Fano $n$-fold $X$ such that 
\begin{enumerate}
\renewcommand{\theenumi}{\arabic{enumi}}
\renewcommand{\labelenumi}{$(\theenumi)$}
\item\label{Aubincounter1}
the anticanonical 
volume of $X$ is equal to $2(3^n-1)$ 
$($note that $2(3^n-1)<\left(\left(n+1\right)^2/2n\right)^n$ if $n\geq 5$$)$ 
and 
\item\label{Aubincounter2}
$X$ does not admit K\"ahler-Einstein metrics.
\end{enumerate} 
\end{corollary}

\begin{proof}
Let $\tau\colon Z\rightarrow\pr^{n-1}$ be the double cover such that 
the branch locus $B\subset\pr^{n-1}$ is 
a smooth divisor of degree $2(n-2)$. We note that $Z$ is isomorphic to a weighted hypersurface of degree $2(n-2)$ 
in $\pr(1^n,n-2)$. 
Let $\sO_Z(1):=\tau^*\sO_{\pr^{n-1}}(1)$, then we have 
$\sO_Z(-K_Z)\simeq
\tau^*\left(\sO_{\pr^{n-1}}\left(-K_{\pr^{n-1}}\right)\otimes\sO_{\pr^{n-1}}\left(2-n\right)\right)\simeq\sO_Z(2).$
Hence $Z$ is a Fano $(n-1)$-fold with $\rho_Z=1$, the Fano index of $Z$ is 
equal to $2$ and $\vol_Z(\sO_Z(1))=2$ holds.

Let $X:=\pr_Z(\sO_Z\oplus\sO_Z(1))$, then $X$ is a Fano $n$-fold and 
$\vol_X(-K_X)=2(3^n-1)$. 
On the other hand, $X$ is not slope semistable by Proposition \ref{nonsemiprop}. 
Thus $X$ does not admit K\"ahler-Einstein metrics by Theorem \ref{KE_stable}.
\end{proof}

\begin{remark}\label{Aubinrmk}
In \cite{aub}, Aubin reduced Conjecture \ref{Aubin} to \cite[Inequality (4)]{aub}. However the inequality does not hold, as already pointed out by
Yuji Sano, for example for $S_3\times \pr^1$.
\end{remark}

\begin{remark}\label{divptfano}
The above Fano manifolds, which is given by $X=\pr_Z(\sO_Z\oplus\sO_Z(s))$ such that 
$Z$ is a Fano $(n-1)$-fold of $\rho_Z=1$ and the Fano index $t$ 
which satisfies $t>s>0$, are characterized by 
the smooth projective varieties which have an elementary birational $K_X$-negative extremal divisor-to-point contraction 
and have a $\pr^1$-bundle structure. See \cite[Remark 2.4 (a)]{fuj} 
or \cite[Lemma 3.6]{casdru}.  
\end{remark}

\section{Threefold case}

Throughout this section, let $X$ be a Fano threefold which satisfies that 
$\xi(D)\leq 0$ for some divisor $D\subset X$. 
For the \emph{type} of an extremal ray for smooth projective threefolds,  
we refer the readers to \cite{MoMu83}.

\subsection{$\rho_X=1$ case}
This case has been shown in Theorem \ref{ample} \eqref{ample1} since any effective divisor is ample. 
We have $X\simeq\pr^3$ and $D$ is a hyperplane section. 
In this case, $X$ is slope semistable along $D$.

\subsection{$\rho_X=2$ case}
We set $\NE(X)=R_1+R_2$ and we also set 
minimal rational curves $[l_1]\in R_1$ and $[l_2]\in R_2$. 
We denote the contractions 
$\phi_i:=\cont_{R_i}\colon X\rightarrow Y_i$ and let $H_i\in\Pic(X)$ 
be the pullback of the ample generator of $\Pic(Y_i)$. 
We note that $\Nef(X)=\R_{\geq 0}[H_1]+\R_{\geq 0}[H_2]$.
Then we have
\begin{itemize}
\item
$\Pic(X)=\Z[H_1]\oplus\Z[H_2]$,
\item
$(H_1\cdot l_2)=1$, $(H_2\cdot l_1)=1$,
\item
$-K_X\sim l(R_2)H_1+l(R_1)H_2$
\end{itemize}
by \cite[Theorem 5.1]{MoMu83}. 

First, we consider the case $l(R_1)=3$ (i.e., $\phi_1$ is a $\pr^2$-bundle). 
Then $X$ is either isomorphic to $\pr^1\times\pr^2$ or $\pr_{\pr^1}(\sO\oplus\sO\oplus\sO(1))$. 
\begin{enumerate}
\renewcommand{\theenumi}{\arabic{enumi}}
\renewcommand{\labelenumi}{$(\theenumi)$}
\item
If $X\simeq\pr^1\times\pr^2$, then $X$ is not slope stable along some divisor 
but slope semistable along any divisor by Theorem \ref{ample} \eqref{ample1} 
and Corollary \ref{prod_cor}.
\item
If $X\simeq\pr_{\pr^1}(\sO\oplus\sO\oplus\sO(1))$, then $X$ is isomorphic to 
the blowing up of $\pr^3$ along a line. Thus $X$ is not slope semistable along the 
exceptional divisor by Propositions \ref{proj} and \ref{reduction}. 
\end{enumerate}

Hence we can assume $l(R_1)\leq 2$ and $l(R_2)\leq 2$. By Proposition \ref{ray}, we can assume 
$l(R_1)=2$ and $(D\cdot l_1)=1$. 
Hence we can write $D\sim aH_1+H_2$ $(a\in\Z)$. 
Note that $a\leq 0$ by Theorem \ref{ample} \eqref{ample1}. 

Assume that $a=0$. We set $b:=l(R_2)$ (note that $b=1$ or $2$). Then we have 
$-K_X\sim bH_1+2H_2$ and $D\sim H_2$ hence $\epsilon(D)=2$. However we have 
\begin{eqnarray*}
\frac{1}{3}\xi(D) & = & \int_0^2(1-x)(H_2.(bH_1+(2-x)H_2)^2)dx\\
& = & \frac{16b}{3}(H_1\cdot H_2^2)+\frac{4}{3}(H_2^3)\geq 0,
\end{eqnarray*}
and equality holds if and only if $(H_1\cdot H_2^2)=0$ and $(H_2^3)=0$. 
In this case $\phi_2$ is a del Pezzo fibration with $l(R_2)\leq 2$ and $l(R_1)=2$. 
However there are no Fano threefolds satisfying these conditions
by \cite[Theorem 1.7]{MoMu83}. 
Therefore $\xi(D)$ always takes a positive value; this leads to a contradiction.

As a consequence, we have $a<0$. Since $D\sim aH_1+H_2$ is effective, $\phi_2$ is a divisorial contraction. 
Hence $R_2$ is of type $E_1$, $E_2$, $E_3$, $E_4$ or $E_5$. 

\begin{enumerate}
\renewcommand{\theenumi}{\arabic{enumi}}
\renewcommand{\labelenumi}{$(\theenumi)$}
\item
If $R_2$ is of type $E_2$, $E_3$, $E_4$ or $E_5$ (divisor-to-point type), 
$X$ is either isomorphic to 
$\pr_{\pr^2}(\sO\oplus\sO(1))$ or $\pr_{\pr^2}(\sO\oplus\sO(2))$ by \cite[Theorem 1.7]{MoMu83}. 
These are not slope semistable along a divisor by Proposition \ref{nonsemiprop}. 
\item
We consider the case that $R_2$ is of type $E_1$(divisor to smooth curve). 
Let $F$ be the exceptional divisor of $\phi_2$ and $t$ be the Fano index of $Y_2$. 
We have $F\sim-H_1+(t-2)H_2$ since $-K_X\sim H_1+2H_2$ and $-K_X\sim tH_2-F$. 
Since $\Eff(X)\cap(\R_{\geq 0}[-H_1]+\R_{\geq 0}[H_2])=\R_{\geq 0}[F]+\R_{\geq 0}[H_2]$ and 
$D\sim aH_1+H_2$ is an effective divisor, we have $t=3$ (i.e., $Y_2\simeq \Q^3$) and $a=-1$ (i.e., $D=F$) 
(hence $\epsilon(D)=2$). 
Therefore $X$ is isomorphic to either $\Bl_{\rm{conic}}\Q^3$ or $\Bl_{\rm{line}}\Q^3$ since $l(R_1)=2$ 
(see \cite[(5.3), (5.5)]{MoMu83}). 
\begin{itemize}
\item
If $X\simeq \Bl_{\rm{conic}}\Q^3$, then it is easy to show that 
$F\simeq\pr^1\times\pr^1$ and $\sN_{F/X}\simeq\sO_{\pr^1\times\pr^1}(2,-1)$ and 
$\sO_X(-K_X)|_F\simeq\sO_{\pr^1\times\pr^1}(4,1)$. 
Hence we have 
\[\frac{1}{3}\xi(F)=\int_0^2(1-x)\vol_{\pr^1\times\pr^1}(\sO(4,1)-x\sO(2,-1))dx=
\frac{8}{3}>0,
\]
this lead to a contradiction.
\item
If $X\simeq\Bl_{\rm{line}}\Q^3$, then it is easy to show that 
$F\simeq\F_1$ and $\sN_{F/X}\simeq \sO_{\F_1}(-e)$ and $-K_X|_F\simeq\sO_{\F_1}(3f+e)$. 
Hence we have 
\[
\frac{1}{3}\xi(F)=\int_0^2(1-x)\vol_{\F_1}(3f+e-x(-e))dx=-\frac{4}{3}<0.
\]
Therefore $X$ is not slope semistable along $F$.
\end{itemize}
\end{enumerate}

\subsection{$\rho_X=3$ case}
By Proposition \ref{ray}, there exists an extremal ray $R\subset\NE(X)$ 
with a minimal rational curve $[C_R]\in R$ such that $l(R)=2$ and $(D.C_R)=1$. 
Hence $R$ is either of type $E_2$ or $C_2$.

(1)
If $R$ is of type $E_2$ (smooth point blowing up), then 
$X$ is isomorphic to $\Bl_p(Y_d)$ (let $\phi:=\Bl_p$) 
with $1\leq d\leq 3$, where $\sigma\colon Y_d\rightarrow\pr^3$ is the blowing up 
of $\pr^3$ along $B$ such that 
$H_0\subset\pr^3$ is a hyperplane, $B\subset H_0$ is a smooth curve of degree $d$, 
$H\subset Y_d$ is the strict transform of $H_0$
and satisfies $p\notin H$ (see \cite[p.\ 160]{MoMu} or \cite{BCW}).

Let $E$ be the exceptional divisor of $\phi$, 
let $F$ be the exceptional divisor of $\sigma$ and 
let $F'$ be the strict transform of the locus of lines passing through $p$ and $B$. 
We set $e\subset E$ and $h\subset H$ such that lines 
(both $E$ and $H$ are isomorphic to $\pr^2$), 
$f\subset F$ be an exceptional curve 
of $\sigma$ and $f'\subset F'$ be the strict transform of a line passing through $p$ and a point in $B$.
Then it is easy to show that 
\begin{eqnarray*}
\NE(X) & = & \R_{\geq 0}[e]+\R_{\geq 0}[h]+\R_{\geq 0}[f]+\R_{\geq 0}[f'],\\
\Pic(X) & = & \Z[E]\oplus\Z[H]\oplus\Z[F],\\
-K_X & \sim & -2E+4H+3F,\\
F' & \sim & -dE+dH+(d-1)F.
\end{eqnarray*}
We can show that 
$E+F'$, $F+H$ and $(F+F')/d$ are nef. Therefore, for $[pE+qH+rF]\in\Eff(X)$ 
($p$, $q$, $r\in\R$), 
we have 
\begin{itemize}
\item
$r=(pE+qH+rF\cdot (F+H)^2)\geq 0,$
\item
$p+r=(pE+qH+rF\cdot ((F+F')/d)^2)\geq 0,$
\item
$dq=(pE+qH+rF\cdot F+H\cdot E+F')\geq 0,$ 
\item
$(d-1)p+dq=(pE+qH+rF\cdot (F+F')/d\cdot E+F')\geq 0.$
\end{itemize}
Hence we have 
\[
\Eff(X)=\R_{\geq 0}[E]+\R_{\geq 0}[H]+\R_{\geq 0}[F]+\R_{\geq 0}[F'].
\]
We write $D\sim pE+qH+rF$, where $p$, $q$, $r\in\Z$. Then we have 
$(D\cdot e)=1$, $(D\cdot f)\leq 0$, $(D\cdot f')\leq 0$ 
and $(D\cdot h)<(-K_X\cdot h)=4-d$ by Proposition \ref{ray}. 
Thus we have $p=-1$, $q=1$, $r=1$ since $D$ is effective. Hence $D\sim -E+H+F$ and $\epsilon(D)=2$. 
Therefore we have 
\begin{eqnarray*}
\frac{1}{3}\xi(D) & = & \int_0^2(1-x)(-E+H+F\cdot ((x-2)E+(4-x)H+(3-x)F)^2)dx\\
& = & \int_0^2(1-x)(-4x+12-d)dx=\frac{8}{3}>0;
\end{eqnarray*}
this leads to a contradiction.  

(2)
If $R$ is of type $C_2$, then $R$ induces a $\pr^1$-bundle $\pi\colon X\rightarrow Z$. 
Since $\rho_X=3$, $Z$ is isomorphic to either $\F_1$ or $\pr^1\times\pr^1$. 

We claim that such Fano threefolds has been classified by Szurek and Wi\'sniewski \cite{SW}: 

\begin{claim}\label{rho3}
\begin{enumerate}
\renewcommand{\theenumi}{\roman{enumi}}
\renewcommand{\labelenumi}{\rm{(\theenumi)}}
\item
If $Z\simeq\F_1$, $X$ is isomorphic to one of $\F_1\times_{\pr^2}\pr(T_{\pr^2})$, 
$\F_1\times\pr^1$ or $\F_1\times_{\pr^2}\pr(\sO\oplus\sO(1))$. 
\item
If $Z\simeq\pr^1\times\pr^1$, $X$ is isomorphic to one of a smooth divisor 
of tridegree $(1,1,1)$ in $\pr^1\times\pr^1\times\pr^2$, $\pr_{\pr^1\times\pr^1}(\sO(0,1)\oplus\sO(1,0))$, 
$\pr^1\times\pr^1\times\pr^1$, 
$\F_1\times\pr^1$ or $\pr_{\pr^1\times\pr^1}(\sO\oplus\sO(1,1))$.  
\end{enumerate}
\end{claim}

(I)
Assume $X\simeq\F_1\times_{\pr^2}\pr(T_{\pr^2})$. Then we can show that $X\subset\F_1\times\pr^2$ is a smooth divisor 
with $X\in|\sO_{\F_1\times\pr^2}(e+f,1)|$. 
Let $E$, $F$, $H$ be effective divisors on $X$ correspond to 
$\sO_X(e,0)$, $\sO_X(f,0)$, $\sO_X(0,1)$, respectively. 
Then we can show that 
\[
\Pic(X)=\Z[E]\oplus\Z[F]\oplus\Z[H].
\]
We can also show that there exists the structure of the blowing up 
$X\rightarrow\pr^1\times\pr^2$ with 
the exceptional divisor $E'\sim H-E$. 
We note that $F$, $H$ and $E+F$ are nef. 
Therefore, for $[pE+qF+rH]\in\Eff(X)$, we have 
\begin{itemize}
\item
$r=(pE+qF+rH\cdot (E+F)^2)\geq 0$,
\item
$q=(pE+qF+rH\cdot H^2)\geq 0$,
\item
$p+r=(pE+qF+rH\cdot F\cdot H)\geq 0.$
\end{itemize}
Hence we have
\[
\Eff(X)=\R_{\geq 0}[E]+\R_{\geq 0}[F]+\R_{\geq 0}[E']
\]
and it is easy to show that $-K_X\sim E+2F+2H$. 

Let $m$ be a fiber of $\pi$, let 
$l$ be an exceptional curve of $X\rightarrow\pr(T_{\pr^2})$ and 
let $l'$ be an exceptional curve of $X\rightarrow\pr^1\times\pr^2$. 
Then it is easy to show that 
\[
\NE(X)=\R_{\geq 0}[m]+\R_{\geq 0}[l]+\R_{\geq 0}[l'].
\]
We write $D\sim pE+qF+rE'$, where $p$, $q$, $r\in\Z$. Then we have 
$(D\cdot m)=1$, $(D\cdot l)\leq 0$ and $(D\cdot l')\leq 0$ by Proposition \ref{ray}. 
We also note that 
$p=1$, $q=0$, $r=1$ (hence $D\sim H$) and $\epsilon(D)=2$ since $D$ is effective. Therefore we have 
\begin{eqnarray*}
\frac{1}{3}\xi(D) & = & \int_0^2(1-x)(H\cdot(E+2F+(2-x)H)^2)dx\\
& = & \int_0^2(1-x)(\sO(0,1)\cdot\sO(e+2f,2-x)^2\cdot\sO(e+f,1))_{\F_1\times\pr^2}dx\\
& = & \int_0^2(1-x)(1+x)dx=\frac{8}{3}>0;
\end{eqnarray*}
this leads to a contradiction. 

(II)
Assume $X\simeq\F_1\times\pr^1$. Then $X$ is not slope semistable along $p_1^*e$ 
by Propositions \ref{dp} and \ref{prod1}. 

(III)
Assume $X\simeq\F_1\times_{\pr^2}\pr(\sO\oplus\sO(1))$. Let $H$ be the section of $\pi$ 
with normal bundle $\sN_{H/X}\simeq\sO_{\F_1}(-e-f)$, let $E$ be the pullback of $e\subset\F_1$ with respect to $\pi$ and let $F$ be the 
pullback of $f\subset\F_1$ with respect to $\pi$. 
Then we can show that $-K_X\sim 4F+3E+2H$, $\epsilon(H)=2$ and 
\begin{eqnarray*}
\frac{1}{3}\xi(H) & = & \int_0^2(1-x)(H\cdot(4F+3E+(2-x)H)^2)dx\\
& = & \int_0^2(1-x)(1+x)(3+x)dx=-4<0,
\end{eqnarray*}
hence $X$ is not slope semistable along $H$.

(IV)
Assume $X\in|\sO_{\pr^1\times\pr^1\times\pr^2}(1,1,1)|$. Let $H_i$ ($1\leq i\leq 3$) be the restriction of $p_i^*\sO(1)$
to $X$. Then we have 
\[
\Pic(X)=\Z[H_1]\oplus\Z[H_2]\oplus\Z[H_3]
\]
by the theorem of Lefschetz. 
We can show that $-K_X\sim H_1+H_2+2H_3$. We can also show that 
$p_{13}|_X\colon X\rightarrow\pr^1\times\pr^2$ and 
$p_{23}|_X\colon X\rightarrow\pr^1\times\pr^2$ are 
the blowing up along smooth curves with the 
exceptional divisors $F_{13}\sim H_1-H_2+H_3$ 
and $F_{23}\sim -H_1+H_2+H_3$, respectively. 

We note that $H_1$, $H_2$ and $H_3$ are nef. Therefore, for $[a_1H_1+a_2H_2+a_3H_3]\in\Eff(X)$ ($a_1$, $a_2$, $a_3\in\R$), we have 
\begin{itemize}
\item
$a_3=(a_1H_1+a_2H_2+a_3H_3\cdot H_1\cdot H_2)\geq 0,$
\item
$a_1+a_3=(a_1H_1+a_2H_2+a_3H_3\cdot H_2\cdot H_3)\geq 0,$
\item
$a_2+a_3=(a_1H_1+a_2H_2+a_3H_3\cdot H_1\cdot H_3)\geq 0,$
\item
$a_1+a_2=(a_1H_1+a_2H_2+a_3H_3\cdot H_3^2)\geq 0.$
\end{itemize}
Hence we have 
\[
\Eff(X)=\R_{\geq 0}[H_1]+\R_{\geq 0}[H_2]+\R_{\geq 0}[F_{13}]+\R_{\geq 0}[F_{23}].
\]
Let $l_3$, $l_2$, $l_1$ be nontrivial irreducible fibers of $p_{12}|_X, p_{13}|_X, p_{23}|_X$, respectively. 
Then we can show that 
\[
\NE(X)=\R_{\geq 0}[l_1]+\R_{\geq 0}[l_2]+\R_{\geq 0}[l_3].
\]
We write $D\sim a_1H_1+a_2H_2+a_3H_3$, where $a_1$, $a_2$, $a_3\in\Z$. Then 
we have $(D\cdot l_1)\leq 0$, $(D\cdot l_2)\leq 0$ and $(D\cdot l_3)=1$ 
by Proposition \ref{ray}. 
We also know that 
$a_1=0$, $a_2=0$ and $a_3=1$ (hence $D\sim H_3$) and $\epsilon(D)=2$ since $D$ is effective. Hence we have 
\begin{eqnarray*}
\frac{1}{3}\xi(D) & = & \int_0^2(H_3\cdot (H_1+H_2+(2-x)H_3)^2)dx\\
& = & \int_0^2(1-x)(\sO(0,0,1)\cdot\sO(1,1,2-x)^2\cdot\sO(1,1,1))_{\pr^1\times\pr^1\times\pr^2}dx\\
& = & \int_0^2(1-x)(10-4x)dx=\frac{8}{3}>0;
\end{eqnarray*}
this leads to a contradiction. 

(V)
Assume $X\simeq\pr_{\pr^1\times\pr^1}(\sO(0,1)\oplus\sO(1,0))$. Let $E_1$ and $E_2$ be the sections of $\pi$ 
such that the normal bundles are $\sN_{E_1/X}\simeq\sO(-1,1)$ and 
$\sN_{E_2/X}\simeq\sO(1,-1)$, 
and $H_i:=\pi^*p_i^*\sO(1)$ ($i=1$, $2$). 
Let $e_i\subset E_i$ be a fiber of the projection 
$p_j\colon E_i\simeq\pr^1\times\pr^1\rightarrow\pr^1$ (for $\{i$, $j\}=\{1$, $2\}$) 
and $f$ be a fiber of $\pi$. 
Then we can show that $-K_X\sim 3H_1+H_2+2E_2$, 
\begin{eqnarray*}
\NE(X) & = & \R_{\geq 0}[e_1]+\R_{\geq 0}[e_2]+\R_{\geq 0}[f],\\
\Pic(X) & = & \Z[H_1]\oplus\Z[H_2]\oplus\Z[E_1],\\
\Eff(X) & = & \R_{\geq 0}[H_1]+\R_{\geq 0}[H_2]+\R_{\geq 0}[E_1]+\R_{\geq 0}[E_2].
\end{eqnarray*}
Hence we can show that $D\sim E_1$ or $D\sim E_2$ or $D\sim E_1+H_1$ 
(in each case we have $\epsilon(D)=2$). 

If $D\sim E_1+H_1$, then we have 
\begin{eqnarray*}
\frac{1}{3}\xi(D) & = & \int_0^2(1-x)(E_1+H_1\cdot((3-x)H_1+H_2+(2-x)E_1)^2)dx\\
& = & \int_0^2(1-x)(2-x)(4-x)dx=\frac{8}{3}>0;
\end{eqnarray*}
this leads to a contradiction.

If $D\sim E_1$ (or $E_2$), then we have 
\begin{eqnarray*}
\frac{1}{3}\xi(D) & = & \int_0^2(1-x)(E_1\cdot(3H_1+H_2+(2-x)E_1)^2)dx\\
 & = & \int_0^2 2(1-x)(1+x)(3-x)dx=0. 
\end{eqnarray*}
Hence $X$ is slope semistable but not slope stable along $E_1$ (and also along $E_2$). 

(VI)
Assume $X\simeq\pr^1\times\pr^1\times\pr^1$. Then $X$ is slope semistable 
along any divisor but is not slope stable along a fiber of $p_1$ by Theorem \ref{ample} 
\eqref{ample1} and Corollary \ref{prod_cor}. 

(VII)
Assume $X\simeq\pr_{\pr^1\times\pr^1}(\sO\oplus\sO(1,1))$. 
Let $E$ be the section of $\pi$ with the normal bundle 
$\sN_{E/X}\simeq\sO(-1,-1)$. Then we have $\epsilon(E)=2$. Therefore 
$X$ is not slope semistable along $E$ by Proposition \ref{convex} \eqref{convex2}.

\subsection{$\rho_X\geq 4$ case}
There exists an extremal ray $R\subset\NE(X)$ of type $C_2$ by Proposition \ref{ray} 
and \cite[p.\ 160]{MoMu}. 
We write its contraction $\pi\colon X\rightarrow S$. 
We know that $S$ is a del Pezzo surface of $\rho_S\geq 3$. 
Hence we have $X\simeq\pr^1\times S_m$ with $1\leq m\leq 7$ by \cite[Theorem 4.20]{MoMu85} (see also \cite{SW}). 

Hence $X$ is slope semistable along any divisor but is not slope stable along some divisor 
by Proposition \ref{dp}, Theorem \ref{ample} \eqref{ample1} and Corollary \ref{prod_cor}.

As a consequence, we have the following result:

\begin{thm}\label{three}
Let $X$ be a Fano threefold.
\begin{enumerate}
\renewcommand{\theenumi}{\arabic{enumi}}
\renewcommand{\labelenumi}{$(\theenumi)$}
\item
$X$ is slope semistable along any effective divisor but 
there exists a divisor $D\subset X$ such that $X$ is not slope stable along $D$ 
if and only if $X$ is isomorphic to one of:
\[
{\pr}^3, {\pr}^1\times{\pr}^2, {\pr}_{{\pr}^1\times{\pr}^1}(\sO (0,1)\oplus\sO (1,0)), 
{\pr}^1\times{\pr}^1\times{\pr}^1, {\pr}^1\times S_m\,\, (1\leq m\leq 7).
\]
\item
There exists a divisor $D\subset X$ such that $X$ is not slope semistable along $D$ 
if and only if $X$ is isomorphic to one of:
\begin{eqnarray*}
{\Bl}_{\rm{line}}{\Q}^3, {\Bl}_{\rm{line}}{\pr}^3, {\pr}_{{\pr}^2}(\sO\oplus\sO (1)), 
{\pr}_{{\pr}^2}(\sO\oplus\sO (2)), \\
{\pr}^1\times{\F}_1, {\pr}_{{\F}_1}(\sO\oplus\sO (e+f)), 
{\pr}_{{\pr}^1\times{\pr}^1}(\sO\oplus\sO (1,1)). 
\end{eqnarray*}
\end{enumerate}
\end{thm}

\begin{remark}\label{tangent}
By Theorem \ref{three}, \cite[Theorem 1.1]{fjt} and the result of Steffens 
\cite[Theorem 3.1]{steffens}, 
There exists a Fano threefold $X$ which is slope stable along all divisors and smooth subvarieties 
but has the unstable tangent bundle. For example, $X$ is the blowing up of 
$\pr_{\pr^2}(\sO\oplus\sO(1))$ along a line on the exceptional divisor ($\simeq\pr^2$) 
of the blowing up $\pr_{\pr^2}(\sO\oplus\sO(1))\rightarrow\pr^3$ 
(no.\ 29 in Table 3 in Mori and Mukai's list \cite{MoMu}). 
\end{remark}

\bigskip
\bigskip

\noindent K.\ Fujita

Research Institute for Mathematical Sciences (RIMS),
Kyoto University, Oiwake-cho, 

Kitashirakawa, Sakyo-ku, Kyoto 606-8502, Japan 

fujita@kurims.kyoto-u.ac.jp

\end{document}